\documentclass{article}
\usepackage[top=30truemm, bottom=30truemm, left=30truemm, right=30truemm]{geometry}
\usepackage[dvipdfmx]{graphicx}
\usepackage{amsmath, amssymb}
\usepackage{amsthm}
\usepackage{autobreak}
\usepackage{amsfonts}
\usepackage{mathdots}
\usepackage{mathrsfs}
\usepackage{mathtools}
\usepackage{ascmac}
\usepackage{comment}
\usepackage{proof}
\usepackage{setspace}
\theoremstyle{definition}
\newtheorem{theorem}{Theorem}
\newtheorem{definition}[theorem]{Definition}
\newtheorem{lemma}[theorem]{Lemma}

\newtheorem{proposition}[theorem]{Proposition}
\newtheorem{remark}{Remark}
\usepackage{cases}
\usepackage{amscd}
\usepackage[all]{xy}
\usepackage{color}
\usepackage{here}
\usepackage{bm}
\title{Infinitely many solutions for $p$-fractional Choquard type equations involving general nonlocal nonlinearities with critical growth via the concentration compactness method
}
\author{Masaki Sakuma}

\begin{document}
\maketitle
\begin{abstract}
We prove the existence of infinitely many solutions to a fractional Choquard type equation 
\[
(-\Delta)^s_p u+V(x)|u|^{p-2}u=(K\ast g(u))g'(u)+\varepsilon_W W(x)f'(u)\quad\text{in }\mathbb{R}^N
\]
involving fractional $p$-Laplacian and a general convolution term with critical growth. In order to obtain infinitely many solutions, we use a type of the symmetric mountain pass lemma which gives a sequence of critical values converging to zero for even functionals. To assure the $(PS)_c$ conditions, we also use a nonlocal version of the concentration compactness lemma.
\vspace{1ex}\par
\end{abstract}
\section{Introduction}

In the present paper, we consider a $p$-fractional Choquard type equation of the form
\begin{equation}
(-\Delta)^s_p u+V(x)|u|^{p-2}u=\frac{1}{p_{r;s}^{\uparrow *}}(K\ast g(u))g'(u)+\varepsilon_W W(x)f'(u)\quad\text{in }\mathbb{R}^N,
\label{cho}
\end{equation}
where $s\in (0,1)$, $1<p<N/s <2pr$, $r\in (1,\infty)$. $(-\Delta)^s_p$ denotes the fractional $p$-Laplacian which may be defined up to normalization factors as
\[
(-\Delta)^s_p u(x)\coloneqq 2\lim_{\varepsilon\to +0}\int_{\mathbb{R}^N\setminus B_\varepsilon (x)}\frac{|u(x)-u(y)|^{p-2}(u(x)-u(y))}{|x-y|^{N+ps}}dy.
\]

Here, we introduce two important critical exponents
\[
\displaystyle p_{r;s}^{\downarrow *} =p\left(1-\frac{1}{2r}\right) \quad\text{and}\quad\displaystyle p_{r;s}^{\uparrow *} =\frac{pN(2r-1)}{2r(N-ps)}. 
\]
Let us note that they satisfy
\[
\frac{1}{r}+2\cdot\frac{1}{p/p_{r;s}^{\downarrow *} }=2 \quad\text{and}\quad
\frac{1}{r}+2\cdot\frac{1}{p_s^{*}/p_{r;s}^{\uparrow *}}=2,
\]
that is, 
\[
\frac{p}{p_{r;s}^{\downarrow *}}= \frac{p_s^{*}}{p_{r;s}^{\uparrow *}}=\ell_r\coloneqq\frac{2r}{2r-1},
\]
where $p_s^{*}=\displaystyle\frac{pN}{N-ps}$ denotes the fractional Sobolev critical exponent. Notice that $p<p_{r;s}^{\uparrow *}$ if and only if $N<2psr$. \par
Throughout this paper, $C$ and $C_i$ denote generic positive constants, $\|\cdot\|_q$ denotes the $L^q$ norm, $L^q_{+}= L^q_{+} (\mathbb{R}^N)$ denotes the set consisting of all positive $L^q$ functions for $q\in [1,\infty]$, $\mu_\mathcal{L}$ denotes the Lebesgue measure in $\mathbb{R}^N$, and $B_r(x)$ denotes the open ball with radius $r$ centered at $x$ in $\mathbb{R}^N$. \par
For a domain $\Omega\subset\mathbb{R}^N$ and $q\in [1,\infty)$, $L^{q,\infty}(\Omega)$ denotes the weak Lebesgue space whose integrability exponent is $q$, which is defined as follows:
\begin{align*}
L^{q,\infty}(\Omega)&\coloneqq\{f\mid \text{$f$ is measurable in $\Omega$}, \sup_{t>0} t\lambda_f (t)^{1/q}<\infty\} \\
&= \{f\mid \text{$f$ is measurable in $\Omega$}, \|f\|_{L^{q,\infty}(\Omega)}\coloneqq\sup_{s>0} s^{1/q} f^*(s)<\infty\},
\end{align*}
where
\[
\lambda_f (t)\coloneqq \mu_{\mathcal{L}}(\{x\in\Omega\mid |f(x)|>t \})\quad (t>0)
\]
is the distribution function of $f$, and
\[
f^*(s)\coloneqq \inf\{t>0\mid \lambda_f (t)\leq s\}
\]
is the (nonsymmetric) decreasing rearrangement of $f$. 
We simply write $L^{q,\infty}= L^{q,\infty}(\mathbb{R}^N)$ and define
\[
L^{q,\infty}_\mathrm{loc}(\Omega)\coloneqq\{f\mid \text{$f$ is measurable in $\Omega$}, f|_{\Omega'}\in L^{q,\infty}(\Omega')\;(\forall \Omega':\text{compact subset of $\Omega$})\}.
\]
For more details about the weak Lebesgue spaces, see \cite{Classical Fourier Analysis}. \par
In the equation \eqref{cho}, $K\in L^{r,\infty}(\mathbb{R}^N)$ and $f,g\in C^1(\mathbb{R};[0,\infty))$ satisfy the following conditions:
\begin{enumerate}
\item[(G1)] $h(-u)=h(u)$, $h(0)=0$ and $|h'(u)|\leq C(|u|^{\hat{p}_g-1}+|u|^{p_g-1})$ ($\forall u\in\mathbb{R}$) for some $p_g,\hat{p}_g$ with $(p_{r;s}^{\downarrow *} <)\, p<\hat{p}_g\leq p_g<p_{r;s}^{\uparrow *}$ and $C>0$, where $h(t)\coloneqq g(t)-|t|^{p_{r;s}^{\uparrow *}}$.
\item[(G2)] $0<\alpha_g h(u)\leq u h'(u)$ ($\forall u\neq 0$) for some $\alpha_g\in (p_{r;s}^{\downarrow *}, p_{r;s}^{\uparrow *}]$. 
\item[(F)] $f(-u)=f(u)$, $f(0)=0$ and $0<\alpha_f f(u)\leq u f'(u)\leq C(|u|^{q_1}+|u|^{q_2})$ ($\forall u\in\mathbb{R}\setminus\{0\}$) for some $q_1,q_2\in (1,p)$, $\alpha_f>1$ and $C>0$.
\item[(K)] $K\in L^{r,\infty}_{+}(\mathbb{R}^N)$ and $K|_{\mathbb{R}^N\setminus B_\varepsilon (0)}\in L^{\infty}(\mathbb{R}^N\setminus B_\varepsilon (0))$ for any $\varepsilon>0$. 
\end{enumerate}
The typical example of $K$ is the Riesz potential $I_\alpha$ (for $\alpha=N/r$) defined by
\[
I_\alpha(x)=\frac{c_{N,\alpha}}{|x|^\alpha},\quad c_{N,\alpha}=\frac{\Gamma(\alpha/2)}{2^{N-\alpha}\pi^{N/2}\Gamma((N-\alpha)/2)}.
\]
However, let us notice that $N-\alpha$ takes the place of $\alpha$ in the definition of $I_\alpha$ in some literature. \par
Let us define the best constant
\[
S_{D^{s,p}}\coloneqq\inf_{u\in D^{s,p}(\mathbb{R^N})\setminus\{0\}}\frac{\|u\|_{D^{s,p}}^p}{\|u\|_{p_s^{*}}^p}
\]
for the Sobolev embedding $D^{s,p}(\mathbb{R}^N)\hookrightarrow L^{p_s^*} (\mathbb{R}^N)$, where $D^{s,p}= D^{s,p} (\mathbb{R}^N)$ is the completion of the space $C_c^\infty (\mathbb{R}^N)$ consisting of smooth functions with compact supports with respect to
\[
\|u\|_{D^{s,p}}\coloneqq \int_{\mathbb{R}^N} \int_{\mathbb{R}^N}\frac{|u(x)-u(y)|^p}{|x-y|^{N+ps}}dxdy.
\]
Especially, in the case $p=2$, we have a simple explicit formula
\[
\displaystyle\inf_{u\in D^{s,2}(\mathbb{R^N})\setminus\{0\}}\frac{\int_{\mathbb{R}^N}|(-\Delta)^{s/2}u|^2 dx}{(\int_{\mathbb{R}^N}|u|^{2^{*}_s} dx)^{2/2^{*}_s}}=\frac{(4\pi)^{s}\Gamma(N/2+s)\Gamma(N/2)^{2s/N}}{\Gamma(N/2-s)\Gamma(N)^{2s/N}}.
\]
We also define another important best constant
\[
S_{K}\coloneqq\inf_{u\in D^{s,p}\setminus\{0\}}\frac{\|u\|_{D^{s,p}}^p}{\left(\int_{\mathbb{R}^N}(K\ast |u|^{p_{r;s}^{\uparrow *}}) |u|^{p_{r;s}^{\uparrow *}}dx\right)^{p/(2\cdot p_{r;s}^{\uparrow *})}}.
\]
In the equation \eqref{cho}, $V:\mathbb{R}^N\to\mathbb{R}$ and $W:\mathbb{R}^N\to\mathbb{R}$ satisfy the following conditions:
\begin{enumerate}
\item[(V)] $V_{+}\in L^\infty_\mathrm{loc}(\mathbb{R}^N)$, $\exists\tau_0>0$ s.t. $\mu_\mathcal{L}(V^{-1}((-\infty,\tau_0]))<\infty$ and $\|V_{-}\|_{N/(ps)}<S_{D^{s,p}}$, where $\mu_\mathcal{L}$ denotes the Lebesgue measure.
\item[(W)] $W\in L^{\frac{p_s^*}{p_s^*-q_1}}_{+}(\mathbb{R}^N)\cap L^{\frac{p_s^*}{p_s^*-q_2}}_{+}(\mathbb{R}^N)$.
\end{enumerate}

As for (V), let us note that $N/(ps)=(p_s^*/p)'$, where ${}'$ denotes the H\"{o}lder conjugate exponent. Therefore, the mapping $\displaystyle u\mapsto \int_{\mathbb{R}^N}V_{-}|u|^p dx$ is continuous in $W^{s,p}(\mathbb{R}^N)$. On the other hand, the mapping $\displaystyle u\mapsto \int_{\mathbb{R}^N}V_{+}|u|^p dx$ is not necessarily continuous in $W^{s,p}(\mathbb{R}^N)$. We introduce a norm
\[
\|u\|_{s,p,V_{+}}\coloneqq \left(\|u\|_{D^{s,p}}^p+ \int_{\mathbb{R}^N}V_{+}|u|^p dx\right)^{1/p}
\]
and a uniformly convex Banach space
\[
E\coloneqq \{u\in W^{s,p}(\mathbb{R}^N)\mid\|u\|_{s,p,V_{+}}<\infty\}
\]
equipped with the norm $\|\cdot \|_{s,p,V_{+}}$, which is also denoted by $\|\cdot\|_E$ or $\|\cdot\|$ for simplicity below. We work with this function space and consider some continuous functionals in this space. The proof of the uniform convexity of $E$ is similar to that of \cite{UniformlyConvex}, so we omit it. In particular, $E$ is reflexive. \par
The action functional associated with the equation \eqref{cho} is
\begin{align*}
I[u] &=\frac{1}{p}\|u\|_{D^{s,p}}^p+ \frac{1}{p}\int_{\mathbb{R}^N}V(x)|u|^p dx \\
&\quad\quad -\frac{1}{2\cdot p_{r;s}^{\uparrow *}} \int_{\mathbb{R}^N}(K\ast g(u))g(u)dx-\varepsilon_W\int_{\mathbb{R}^N}W(x)f(u)dx.
\end{align*}
The statement of our main theorem is as follows.
\begin{theorem}\label{MainTheorem}
Under the conditions (G1), (G2), (V), (K), (F) and (W), for each $\varepsilon_W>0$ small enough, there exists a sequence $\{u_n\}$ of the solutions for \eqref{cho} with $I[u_n]<0$ ($\forall n$) and $I[u_n]\to 0$ ($n\to\infty$).
\end{theorem}
In recent years, the study of Choquard type equations and fractional versions of elliptic equations has been attracting a lot of attention. The nonlinear Choquard equation 
\[
-\Delta u+V(x)u=(I_\alpha \ast |u|^q)|u|^{q-2}u\quad\text{in }\mathbb{R}^N
\]
arose in the description of the quantum theory of polaron at rest by Pekar \cite{Pekar} in 1954 and the modeling of an electron trapped in its own hole in the work by Choquard in 1976, which is also related to a Hartree approximation about one-component plasma in \cite{Hartree-Fock}. Regarding the recent results on the (non-fractional) Choquard equations, see \cite{Moroz}. On the other hand, in the context of fractional quantum mechanics, the nonlinear fractional Schr\"{o}dinger equation was first proposed by Laskin \cite{Laskin} as a result of expanding the classical Feynman path integral to the L\'{e}vy-like quantum mechanical paths. The stationary states of the corresponding fractional Schr\"{o}dinger-Newton equations satisfy the fractional Choquard equations. d'Avenia et al. \cite{classical fractional Choquard 1, AttainTalenti} studied the existence and some properties of the weak solutions for the fractional subcritical Choquard equation 
\[
(-\Delta)^{s} u+V(x)u=(I_\alpha \ast |u|^q)|u|^{q-2}u\quad\text{in }\mathbb{R}^N.
\]
As for the critical case, \cite{fractionalCC} studied the existence of high energy solutions for the fractional critical Choquard equation under the assumptions on the potential function $V$ introduced by Benci and Cerami. While most of the works on the fractional or non-fractional Choquard equations deal with the classically typical case where the convolution is the Riesz potential, several works, including \cite{6}, \cite{Bhattarai}, \cite{Qin} and so forth, deal with the convolutions with more general kernels in weak Lebesgue spaces. Besides, especially in recent years, $p$-fractional versions and Kirchhoff type problems have attracted a lot of attention. In \cite{Kirchhoff}, the authors consider the Schr\"{o}dinger-Choquard-Kirchhoff type equation 
\[
M(\|u\|_{D^{s,p}}^p)(-\Delta)_p^s u+V(x)|u|^{p-2}u=\lambda(I_\alpha\ast |u|^{p_{N/\alpha ;s}^{\uparrow *}}) |u|^{p_{N/\alpha ;s}^{\uparrow *}-2} u+\beta k(x)|u|^{q-2}u \quad\text{in }\mathbb{R}^N
\]
involving $p$-Laplace operator, where $M$ is a non-degenerate Kirchhoff function and $V\in C(\mathbb{R}^N)$ satisfies $\inf V>0$, and obtain infinitely many solutions by using the symmetric mountain pass lemma. Let us note that we can also extend Theorem \ref{MainTheorem} to Kirchhoff type equations using the method in \cite{Kirchhoff}. \par
Inspired by the above works, we consider the $p$-fractional Choquard type equation \eqref{cho} involving general convolution potential $K$. The main difficulties in solving our problem are the lack of compactness due to the translation in $\mathbb{R}^N$ and the dilation in $D^{s,p}$ and the double nonlocal nature caused by the fractional $p$-Laplacian and the convolution term. By extending the concentration compactness principle to a $p$-fractional version involving some nonlocal quantities, we overcome these difficulties. In this way, the $(PS)_c$ condition for the energy functional corresponding to the equation \eqref{cho} is shown for $c<0$. Furthermore, by utilizing the symmetric pass lemma based on the concept of Krasnoselskii's genus, we derive the existence of infinitely many critical points. Note that there are two types of the symmetric mountain pass lemma: one type assures the existence of a sequence of critical values diverging to infinity while another provides a sequence of critical values converging to zero. We deal with the latter case focusing on the effect of the term in the functional of the order determined by exponents smaller than $p$. This can be regarded as a generalization of the sublinear nonlinearity of the equation in the case $p=2$.

\section{Preliminaries}
It is easy to check that $I$ is well-defined in $E$ and that we have
\begin{align*}
I'[u]v&=\int_{\mathbb{R}^N} \int_{\mathbb{R}^N} \frac{|u(x)-u(y)|^{p-2}(u(x)-u(y))(v(x)-v(y))}{|x-y|^{N+ps}}dxdy \\
&\quad\quad+ \int_{\mathbb{R}^N}V(x)|u|^{p-2}u v dx-\frac{1}{p_{r;s}^{\uparrow *}} \int_{\mathbb{R}^N}(K\ast g(u))g'(u)v dx -\varepsilon_W\int_{\mathbb{R}^N}W(x)f'(u)v dx.
\end{align*}

To simplify the notation, let us define
\[
J[u]\coloneqq \displaystyle \frac{1}{2\cdot p_{r;s}^{\uparrow *}} \int_{\mathbb{R}^N}(K\ast g(u))g(u)dx,
\]
which is the convolution term of $I[u]$. \par
We describe the concept of weak solutions under the above setting.
\begin{definition}
We say that $u \in E$ is a (weak) solution to the equation \eqref{cho} if $u$ is a critical point of $I$. 
\end{definition}

Using the counterpart of the Young's convolution inequality for weak Lebesgue spaces (see Theorem 1.2.13 in \cite{Classical Fourier Analysis}) and usual H\"{o}lder's inequality, we can easily obtain the following key fact taking the place of the Hardy-Littlewood-Sobolev inequality for $K=I_\alpha$ in the case of general $K\in L^{r,\infty}(\mathbb{R}^N)$. 
\begin{proposition}
Let $p_1,p_2,p_3\in (1,\infty)$ with $1/p_1+1/p_2+1/p_3=2$. There exists a constant $C>0$ such that for all $f\in L^{p_1,\infty}(\mathbb{R}^N)$, $g\in L^{p_2} (\mathbb{R}^N)$, $h\in L^{p_3} (\mathbb{R}^N)$, we have
\[
\|(f\ast g)h\|_1\leq \|f\|_{p_1,\infty}\|g\|_{p_2}\|h\|_{p_3}.
\]
\end{proposition}

The following lemma is a $p$-fractional version of that in \cite{Vminus}, which is for Schr\"{o}dinger equations. 
\begin{lemma} 
There exist constants $C_1,C_2>0$ such that, for any $u\in W^{s,p}(\mathbb{R}^N)$ such that $\|u\|_{s,p,V_{+}}<\infty$, we have
\[
C_1\|u\|_{W^{s,p}}^p\leq C_2\|u\|_{s,p,V_{+}}^p\leq \|u\|_{D^{s,p}}^p+ \int_{\mathbb{R}^N}V|u|^p dx\leq \|u\|_{s,p,V_{+}}^p.
\]
\end{lemma}
\begin{proof}
From the condition (V), there exists $\tau_0>0$ such that $\mu_\mathcal{L}(V^{-1}((-\infty,\tau_0]))<\infty$ where $\mu_\mathcal{L}$ denotes the Lebesgue measure. Since 
\[
\|(V-\tau_0)_{-}\|_{N/(ps)}^{N/(ps)}\leq \|V_{-}\|_{N/(ps)}^{N/(ps)}+\tau_0 \mu_\mathcal{L}(V^{-1}((-\infty,\tau_0]))<\infty, 
\]
we have $(V-\tau_0)_{-}\in L^{N/(ps)}$. Therefore, by the Lebesgue convergence theorem with the dominating function $(V-\tau_0)_{-}^{N/(ps)}$, we get $\|(V-t)_{-}\|_{N/(ps)}\to \|V_{-}\|_{N/(ps)}^{N/(ps)}<S_{D^{s,p}}$ ($t\to 0$). It follows that there exist $\varepsilon_0>0$ and $t_0\in (0,\tau_0)$ such that 
\[
\|(V-t_0)_{-}\|_{N/(ps)}<S_{D^{s,p}}-\varepsilon_0.
\]

Since $(p_s^{*}/p)'=N/(ps)$, for $t>0$, by H\"{o}lder's inequality,
\begin{align*}
\int_{\mathbb{R}^N}V|u|^p dx &= \int_{\mathbb{R}^N}((V-t)_{+}+t)|u|^p dx-\int_{\mathbb{R}^N}(V-t)_{-}|u|^p dx\\
&\geq \int_{\mathbb{R}^N}((V-t)_{+}+t)|u|^p dx-\|(V-t)_{-}\|_{N/(ps)}\|u\|_{p_s^{*}}^p \\
&\geq \int_{\mathbb{R}^N}((V-t)_{+}+t)|u|^p dx-S^{-1}_{D^{s,p}}\|(V-t)_{-}\|_{N/(ps)}\|u\|_{D^{s,p}}^p.
\end{align*}
Therefore, for $t_0>0$ sufficiently small, 
\begin{align*}
\|u\|_{D^{s,p}}^p+ \int_{\mathbb{R}^N}V|u|^p dx &\geq \|u\|_{D^{s,p}}^p+ \int_{\mathbb{R}^N}((V-t_0)_{+}+t_0)|u|^p dx \\
&\quad\quad -S^{-1}_{D^{s,p}}\|(V-t_0)_{-}\|_{N/(ps)}\|u\|_{D^{s,p}}^p \\
&\geq t_0\|u\|_{p}^p+(1-S_{D^{s,p}}^{-1}\|(V-t_0)_{-}\|_{N/(ps)}) \|u\|_{D^{s,p}}^p \\
&\geq t_0\|u\|_{p}^p+ S_{D^{s,p}}^{-1} \varepsilon_0 \|u\|_{D^{s,p}}^p \\
&\geq C\|u\|_{W^{s,p}}^p.
\end{align*}
Replacing $V$ with $V_{+}$, we obtain $\|u\|_{W^{s,p}}\leq C\|u\|_{s,p,V_{+}}$. 
\begin{align*}
\|u\|_{D^{s,p}}^p+ \int_{\mathbb{R}^N}V|u|^p dx &=\|u\|_{s,p,V_{+}}^p-\int_{\mathbb{R}^N}V_{-}|u|^p dx \\
&\geq \|u\|_{s,p,V_{+}}^p-\|V_{-}\|_{N/(ps)} \|u\|_{p_s^{*}}^p \\
&\geq \|u\|_{s,p,V_{+}}^p-S_{D^{s,p}}^{-1}\|V_{-}\|_{N/(ps)} \|u\|_{D^{s,p}}^p \\
&\geq (1-S_{D^{s,p}}^{-1}\|V_{-}\|_{N/(ps)}) \|u\|_{s,p,V_{+}}^p.
\end{align*}
\end{proof}

\begin{lemma}
If $\{u_n\}$ is a $(PS)_c$ sequence for $I$, then $\{u_n\}$ is bounded. 
\end{lemma}
\begin{proof}
Since $I[u_n]\to c$ and $\|I'[u_n]\|_{E'}\to 0$, for sufficiently large $n$, we have
\begin{align*}
&\phantom{=}c+1+\beta \cdot \|u_n\| \\
&\geq I[u_n]-\beta I'[u_n]u_n \\
&=\left(\frac{1}{p}-\beta\right)\left(\|u_n\|_{D^{s,p}}^p+\int_{\mathbb{R}^N}V|u_n|^p dx\right) \\
&\quad\quad -\frac{1}{2\cdot p_{r;s}^{\uparrow *}} \int_{\mathbb{R}^N}(K\ast g(u_n))(g(u_n)-2\beta g'(u_n)u_n)dx -\varepsilon_W\int_{\mathbb{R}^N}W(x)(f(u_n)-\beta f'(u_n)u_n)dx\\
&\geq\left(\frac{1}{p}-\beta\right)\left(\|u_n\|_{D^{s,p}}^p+\int_{\mathbb{R}^N}V|u_n|^p dx\right) \\
&\quad\quad -\frac{1/\alpha_g-2\beta}{2\cdot p_{r;s}^{\uparrow *}} \int_{\mathbb{R}^N}(K\ast g(u_n))g'(u_n)u_n dx -\left(\frac{1}{\alpha_f}-\beta\right)\varepsilon_W\int_{\mathbb{R}^N}W(x)f'(u_n)u_n dx\\
&\geq\left(\frac{1}{p}-\beta\right)\left(\|u_n\|_{D^{s,p}}^p+\int_{\mathbb{R}^N}V|u_n|^p dx\right) \\
&\quad\quad +\frac{2\beta-1/\alpha_g}{2\cdot p_{r;s}^{\uparrow *}} \cdot p_{r;s}^{\uparrow *}\left(\|u_n\|_{D^{s,p}}^p+\int_{\mathbb{R}^N}V|u_n|^p dx 
-\varepsilon_W\int_{\mathbb{R}^N}W(x)f'(u_n)u_n dx +o(1)\right) \\ 
&\quad\quad-\left(\frac{1}{\alpha_f}-\beta\right)\varepsilon_W\int_{\mathbb{R}^N}W(x)f'(u_n)u_n dx\\
&\geq\left(\frac{1}{p}-\beta\right)\left(\|u_n\|_{D^{s,p}}^p+\int_{\mathbb{R}^N}V|u_n|^p dx\right) \\
&\quad\quad +\frac{2\beta-1/\alpha_g}{2\cdot p_{r;s}^{\uparrow *}} \cdot p_{r;s}^{\uparrow *}\left(\|u_n\|_{D^{s,p}}^p+\int_{\mathbb{R}^N}V|u_n|^p dx\right) -C_1\varepsilon_W\sum_{j=1,2}\|W\|_{\frac{p_s^*}{p_s^*-q_j}}\| |u_n|^{q_j}\|_{p_s^*/q_j}-C_2 \\
&\geq C_3\|u_n\|^p-C_4\varepsilon_W(\|u_n\|^{q_1}+ \|u_n\|^{q_2})-C_2
\end{align*}
for some $\beta\in ((2\alpha_g)^{-1},1/p)$ and positive constants $C_j>0$. Since $1<q_i<p$ ($i=1,2$), comparing the order of both sides, we deduce $\{\|u_n\|\}$ is bounded. 

\end{proof}
\begin{remark}\label{uniformboundforW}
$\displaystyle \sup_{\{u_n\}:(PS)_c\text{ sequence for $I$}}\sup_{n\in\mathbb{N}} \|u_n\|$ is bounded uniformly in $\varepsilon_W \in (0,\varepsilon)$ for any $\varepsilon>0$ because we can further estimate as follows:
\begin{align*}
&\phantom{=}C_3\|u_n\|^p-C_4\varepsilon_W(\|u_n\|^{q_1}+ \|u_n\|^{q_2})-C_2 \\
&\geq C_3\|u_n\|^p-C_4\varepsilon (\|u_n\|^{q_1}+ \|u_n\|^{q_2})-C_2.
\end{align*}
\end{remark}

\begin{proposition}
If $\{u_n\}$ is a $(PS)_c$ sequence for $I$ with $u_n \rightharpoonup u$, then $I'[u]=0$. 
\end{proposition}
\begin{proof}
By the fractional-order Rellich–Kondrachov theorem, up to a subsequence, we have $u_n\to u$ in $L^q_\mathrm{loc}(\mathbb{R}^N)$ for any $q\in (p,p_s^*)$ and $u_n\to u$ a.e. 
By the continuity of Nemytskii operator: \[
L^p\cap L^{p_s^{*}}\to L^{p/p_{r;s}^{\downarrow *}}+ L^{p_s^{*}/p_{r;s}^{\uparrow *}} =L^{\frac{2r}{2r-1}};u\mapsto g(u)
\]
and the fact that the linear operator: $L^{\frac{2r}{2r-1}}\to L^{2r}; v\mapsto K\ast v$ is continuous due to the Young's inequality for weak Lebesgue space, we have $\{K\ast g(u_n)\}$ is bounded in $L^{2r}$. Since $2r>1$, the almost everywhere convergence and the boundedness in $L^{2r}$ implies weak convergence $K\ast g(u_n) \rightharpoonup K\ast g(u)$ in $L^{2r}$. Take any $\varphi\in C_c^\infty$.
For $\varepsilon>0$ sufficiently small, since $u_n\to u$ in $L^{\frac{2r}{2r-1}\cdot (p_{r;s}^{\downarrow *}+\varepsilon)}_\mathrm{loc}\cap L^{\frac{2r}{2r-1}\cdot (p_{r;s}^{\uparrow *}-\varepsilon)}_\mathrm{loc}= L^{\frac{2r}{2r-1}\cdot (p_{r;s}^{\uparrow *}-\varepsilon)}_\mathrm{loc}$, by continuity of Nemytskii operator, we have 
$g'(u_n)\to g'(u)$ in $L^{\frac{2r}{2r-1}\cdot\frac{p_{r;s}^{\uparrow *}-\varepsilon}{p_{r;s}^{\uparrow *}-1}}(\operatorname{supp}\varphi)$ and thus $g'(u_n)\varphi\to g'(u)\varphi$ in $L^{\frac{2r}{2r-1}}=L^{(2r)'}$ by H\"{o}lder's inequality. Therefore, $(K\ast g(u_n))g'(u_n)\varphi\to (K\ast g(u))g'(u)\varphi$ in $L^1$. By the density argument, we obtain $J'[u_n] \rightharpoonup J'[u]$. 
\par
As for the linear part of $I'$, the weak-weak continuity follows from the linearity and the norm continuity. As for the local nonlinear part of $I'$, by H\"{o}lder's inequality, noting that
\[
\frac{1}{\frac{p_s^*}{p_s^*-q_i}}+ \frac{1}{\frac{p_s^*}{q_i-1}}+\frac{1}{p_s^*}=1
\]
and using the characterization of the weak convergence and a similar density argument as above, we can verify the weak-weak continuity.
\end{proof}

The next splitting lemma is a simple generalization of Lemma 2.2 in \cite{Split}. 
\begin{lemma}\label{ConvolutionSplit}
Suppose $\{u_n\}$ is bounded in $L^{p_s^*}(\mathbb{R}^N)$ and $u_n\to u$ a.e. in $\mathbb{R}^N$. Then, $J[u_n]-J[u_n-u]\to J[u]$.
\end{lemma}
\begin{proof}
By general Br\'{e}zis-Lieb lemma (or in the same way as Lemma A in \cite{Small}), 
\[
g(u_n)-g(u_n-u)\to g(u)
\]
in $L^{2r/(2r-1)}(\mathbb{R}^N)$. By the continuity of $v\mapsto K\ast v$ assured by the Young's convolution inequality for weak Lebesgue spaces, 
\[
K\ast (g(u_n)-g(u_n-u))\to K\ast g(u)
\]
in $L^{2r}(\mathbb{R}^N)$. On the other hand, since $\{g(u_n-u)\}$ is bounded in $L^{2r/(2r-1)}(\mathbb{R}^N)$ and converges to $0$ almost everywhere, it converges to $0$ weakly in $L^{2r/(2r-1)}(\mathbb{R}^N)$. Therefore, 
\begin{align*}
&\phantom{=}\int_{\mathbb{R}^N} (K\ast g(u_n))g(u_n)dx-\int_{\mathbb{R}^N} (K\ast g(u_n-u))g(u_n-u)dx \\
&= \int_{\mathbb{R}^N} (K\ast (g(u_n)-g(u_n-u)))(g(u_n)-g(u_n-u))dx \\
&\phantom{=}+2 \int_{\mathbb{R}^N} (K\ast (g(u_n)-g(u_n-u))) g(u_n-u)dx \\
&\to  \int_{\mathbb{R}^N} (K\ast g(u))g(u)dx.
\end{align*}
\end{proof}
On the other hand, the Br\'{e}zis-Lieb type splitting property for $\|\cdot\|_{D^{s,p}}^p$ is already well-known. As for such lemmata, see, e.g., \cite{Small}, \cite{fractionalCC}. \par
We prepare the following simple lemma analyzing the behavior of $|x|K(x)$ in order to establish the concentration compactness lemma for the convolution involving weak $L^r$ functions.
\begin{lemma}
Assume $K\in L^{r,\infty}(\mathbb{R}^N)$ does not have local singularities at any points other than the origin, that is, $K|_{\mathbb{R}^N\setminus\{0\}}\in L^\infty_\mathrm{loc}(\mathbb{R}^N\setminus\{0\})$. Define $K'(x)\coloneqq |x|K(x)$. Then, $K'\in L^{r^*,\infty}_\mathrm{loc} (\mathbb{R}^N)$ where $r^*=\displaystyle \frac{Nr}{N-r}$ if $r<N$; while $r^*=\infty$ if $r\geq N$. 
\end{lemma}
\begin{proof}
Since the weak $L^r$ quasi-norms can be expressed as \[
\|K\|_{L^{r,\infty}}=\sup_{s>0} s^{1/r}K^*(s)
\]
in terms of nonsymmetric decreasing rearrangement, we know $K^*(s)=O(s^{-1/r})$ as $s\to +0$ (which quantifies the local singularity of $K$). Consider the case $r<N$ (otherwise we can infer easier). For $\alpha>0$ sufficiently large,
\begin{align*}
\mu_\mathcal{L}(\{x\in B_1(0)\mid |x|K(x)>\alpha\}) &\leq\mu_\mathcal{L}(\{x\in B_1(0)\mid |x|\cdot C |x|^{-N/r}>\alpha\}) \\
&= O(\alpha^{-\frac{Nr}{N-r}})
\end{align*}
as $\alpha\to\infty$. Hence, for $K'(x)\coloneqq |x|K(x)$, we have
\[
(K')^*(s)=O(\inf\{\alpha>0\mid \alpha^{-\frac{Nr}{N-r}}\leq s\})=O(s^{1/N-1/r}).
\]
This implies $\displaystyle\limsup_{s\to +0} s^{1/r^*}(K')^*(s)<\infty$ and thus, together with $K'|_{\mathbb{R}^N\setminus\{0\}}\in L^\infty_\mathrm{loc}(\mathbb{R}^N\setminus\{0\})$, we obtain $K'\in L^{r^*,\infty}_\mathrm{loc} (\mathbb{R}^N)$.
\end{proof}
The following lemma is a variant of the second concentration compactness principle for nonlocal problems involving general convolution with weak $L^r$ function in fractional Sobolev spaces. This lemma plays an important role in proving $(PS)_c$ conditions.
\begin{lemma}\label{CC}
Let $\{u_n\}$ be a bounded sequence in $D^{s,p}$ converging to some $u\in D^{s,p}$ weakly and almost everywhere. Assume $\displaystyle\int_{\mathbb{R}^N}\frac{|u_n(x)-u_n(y)|^p}{|x-y|^{N+ps}}dy\rightharpoonup \mu$, $|u_n|^{p_s^*}\rightharpoonup \nu$, $(K\ast |u_n|^{p_{r;s}^{\uparrow *}}) |u_n|^{p_{r;s}^{\uparrow *}} \rightharpoonup \xi$ in the sense of vague convergence. Define
\begin{align}
\mu_\infty&\coloneqq\lim_{R\to\infty}\limsup_{n\to\infty}\int_{\{|x|\geq R\}} \left(\int_{\mathbb{R}^N}\frac{|u_n(x)-u_n(y)|^p}{|x-y|^{N+ps}}dy\right)dx; \\
\nu_\infty&\coloneqq \lim_{R\to\infty}\limsup_{n\to\infty}\int_{\{|x|\geq R\}} |u_n|^{p_s^*} dx; \\
\xi_\infty&\coloneqq \lim_{R\to\infty}\limsup_{n\to\infty}\int_{\{|x|\geq R\}} (K\ast |u_n|^{p_{r;s}^{\uparrow *}}) |u_n|^{p_{r;s}^{\uparrow *}} dx.
\end{align}
Then, there exist an at most countable set $\mathcal{I}$, a family of points $\{x_i\}_{i\in \mathcal{I}}\subset\mathbb{R}^N$ and families of nonnegative numbers $\{\xi_i\}_{i\in \mathcal{I}}, \{\mu_i\}_{i\in \mathcal{I}}, \{\nu_i\}_{i\in \mathcal{I}}\subset \mathbb{R}$ such that
\begin{align}
\xi &=(K\ast |u|^{p_{r;s}^{\uparrow *}}) |u|^{p_{r;s}^{\uparrow *}}+\sum_{i\in \mathcal{I}}\xi_i\delta_{x_i}; \\
\mu&\geq \int_{\mathbb{R}^N}\frac{|u(x)-u(y)|^p}{|x-y|^{N+ps}}dy +\sum_{i\in \mathcal{I}}\mu_i \delta_{x_i}; \\
\nu &=|u|^{p_s^*}+\sum_{i\in \mathcal{I}}\nu_i\delta_{x_i};
\end{align}
\[
\sum_{i\in \mathcal{I}}\xi_i^{p_s^*/(2 \cdot p_{r;s}^{\uparrow *})}\left(= \sum_{i\in \mathcal{I}}\xi_i^{\ell_r/2}\right)\leq \sum_{i\in \mathcal{I}}\xi_i^{p/(2 \cdot p_{r;s}^{\uparrow *})} <\infty;
\]
\begin{equation}\label{CCquantity}
\mu_i\geq S_K\xi_i^{p/(2\cdot p_{r;s}^{\uparrow *})},\quad \nu_i\geq L_{K}\xi_i^{p_s^*/(2\cdot p_{r;s}^{\uparrow *})},\quad \mu_i\geq S_{D^{s,p}}\nu_i^{p/p_s^*}
\end{equation}
where
\[
L_K\coloneqq\inf_{u\neq 0}\frac{\|u\|_{p_s^*}^{p_s^*}}{\left(\int_{\mathbb{R}^N}(K\ast |u|^{p_{r;s}^{\uparrow *}}) |u|^{p_{r;s}^{\uparrow *}}dx\right)^{p_s^*/(2\cdot p_{r;s}^{\uparrow *})}}
\]
and we have
\begin{align}
\limsup_{n\to\infty}\int_{\mathbb{R}^N} (K\ast |u_n|^{p_{r;s}^{\uparrow *}}) |u_n|^{p_{r;s}^{\uparrow *}}dx&= \int_{\mathbb{R}^N} d\xi+\xi_\infty;\label{mass conservation for xi}\\
\limsup_{n\to\infty}\|u_n\|_{D^{s,p}}^p &= \int_{\mathbb{R}^N} d\mu+\mu_\infty; \label{mass conservation for mu}\\
\limsup_{n\to\infty}\int_{\mathbb{R}^N} |u_n|^{p_s^*} dx &= \int_{\mathbb{R}^N} d\nu+\nu_\infty; \label{mass conservation for nu}
\end{align}
\begin{align}
L_K^2\xi_\infty^{p_s^*/p_{r;s}^{\uparrow *}}& \leq \nu_\infty\left(\int_{\mathbb{R}^N} d\nu+\nu_\infty\right); \label{xi and nu at infty}\\
S_K^2 \xi_\infty^{p/p_{r;s}^{\uparrow *}}& \leq \mu_\infty\left(\int_{\mathbb{R}^N} d\mu+\mu_\infty\right); \label{xi and mu at infty}\\
S_{D^{s,p}}\nu_\infty^{p/p_s^*}&\leq\mu_\infty. \label{mugeqnuatinfty}
\end{align}
Moreover, \eqref{CCquantity} holds with $\xi(\mathbb{R}^N),\mu(\mathbb{R}^N),\nu (\mathbb{R}^N)$ instead of $\xi_i,\mu_i,\nu_i$. 
\end{lemma}
\begin{proof}
For $v_n\coloneqq u_n-u$, by the Brezis-Lieb type splitting properties, in the vague topology, we have
\[
\int_{\mathbb{R}^N}\frac{|v_n(x)-v_n (y)|^p}{|x-y|^{N+ps}}dy\to \tilde{\mu}\coloneqq \mu -\int_{\mathbb{R}^N}\frac{|u(x)-u(y)|^p}{|x-y|^{N+ps}}dy,
\]
\[
(K\ast |v_n|^{p_{r;s}^{\uparrow *}})|v_n|^{p_{r;s}^{\uparrow *}}\to \tilde{\xi}\coloneqq\xi -(K\ast |u|^{p_{r;s}^{\uparrow *}})|u|^{p_{r;s}^{\uparrow *}},
\]
\[
|v_n|^{p_{r;s}^{\uparrow *}}\to\tilde{\nu}\coloneqq\nu-|u|^{p_{r;s}^{\uparrow *}}.
\]
Take any $\varphi\in C_c^\infty$. We define $\kappa_\varphi (x,y)\coloneqq K(x-y)(|\varphi (y)|^{p_{r;s}^{\uparrow *}}-|\varphi (x)|^{p_{r;s}^{\uparrow *}})$. 
Then, 
\begin{equation}\label{Gamma_n}
\begin{split}
&\phantom{=}\left| \int_{\mathbb{R}^N}(K\ast |\varphi v_n|^{p_{r;s}^{\uparrow *}}) |\varphi v_n|^{p_{r;s}^{\uparrow *}}dx -\int_{\mathbb{R}^N}(K\ast |v_n|^{p_{r;s}^{\uparrow *}}) | \varphi |^{p_{r;s}^{\uparrow *}}|\varphi v_n|^{p_{r;s}^{\uparrow *}}dx\right| \\
&= \left| \int_{\mathbb{R}^N}\Gamma_n (x)dx\right| = \left| \int_{x\in\operatorname{supp}\varphi} \Gamma_n (x) dx\right|
\end{split}
\end{equation}
where
\begin{align*}
&\phantom{=}\Gamma_n(x) \\
&\coloneqq \displaystyle\int_{\mathbb{R}^N}\kappa_\varphi(x,y) |v_n(y)|^{p_{r;s}^{\uparrow *}} dy|\varphi(x) v_n(x)|^{p_{r;s}^{\uparrow *}} \\
&=\left(\int_{|y|\leq R}\kappa_\varphi(x,y) |v_n(y)|^{p_{r;s}^{\uparrow *}} dy- |\varphi |^{p_{r;s}^{\uparrow *}} \int_{|y|> R}K(x-y)|v_n(y)|^{p_{r;s}^{\uparrow *}} dy\right)|\varphi v_n|^{p_{r;s}^{\uparrow *}}.
\end{align*}
Note that
\[
\kappa_\varphi(x,y)=|x-y|
K(x-y)\cdot \frac{|\varphi (y)|^{p_{r;s}^{\uparrow *}}-|\varphi (x)|^{p_{r;s}^{\uparrow *}}}{|x-y|}
\]
and
\[
\displaystyle \frac{|\varphi (y)|^{p_{r;s}^{\uparrow *}}-|\varphi (x)|^{p_{r;s}^{\uparrow *}}}{|x-y|}\in L^{\infty}(\mathbb{R}^{N}\times \mathbb{R}^{N}).
\]
Since $[z\mapsto K'(z)=|z|K(z)]\in L^{r^*,\infty}_\mathrm{loc}$, we notice that $K'\in L^q (\operatorname{supp}\varphi -B_R(0))$ and $\{\kappa_\varphi (x,\cdot)\}_{x\in\mathbb{R}^N}$ is bounded in $L^q (B_R(0))$ for $q\in [1,r^*)$ when $r<N$ while $q\in [1,\infty]$ when $r\geq N$. By the Young's inequality, 
\begin{align*}
\left\|\int_{B_R(0)}\kappa_\varphi (x,y) |v_n(y)|^{p_{r;s}^{\uparrow *}}dy\right\|_{L^\sigma_x(\operatorname{supp}\varphi)}& \leq C_1 \|K'\ast |v_n|^{p_{r;s}^{\uparrow *}}\|_{L^\sigma (\operatorname{supp}\varphi)} \\
&\leq C_1\|K'\|_{L^q (\operatorname{supp}\varphi -B_R(0))} \||v_n|^{p_{r;s}^{\uparrow *}}\|_{\ell_r} \\
&\leq C_2
\end{align*}
for some $q\in [1,r^*)$ and $\sigma> 2r$ with
\[
\frac{1}{q}+\frac{1}{\ell_r}=1+\frac{1}{\sigma}.
\]
Such exponents actually exist because $q=r$ when $\sigma=2r$ and because $1\leq r< r^*$. On the other hand, since $\displaystyle C_{\varphi,R}\coloneqq\operatorname{ess\,sup}\displaylimits_{x\in \operatorname{supp}\varphi,y\not\in B_R(0)}K(x-y)<\infty$ for $R>0$ sufficiently large, we have
\begin{align*}
&\phantom{=}\left\| |\varphi |^{p_{r;s}^{\uparrow *}}\int_{\mathbb{R}^N\setminus B_R(0)}K(x-y)|v_n(y)|^{p_{r;s}^{\uparrow *}}dy\right\|_{L^\sigma_x(\operatorname{supp}\varphi)} \\
& \leq \|K\wedge C_{\varphi,R}\|_{\tilde{r}} \| |v_n|^{p_{r;s}^{\uparrow *}}\|_{\ell_r}\| |\varphi |^{p_{r;s}^{\uparrow *}}\|_{q_\varphi}\leq C_3
\end{align*}
for $\tilde{r}\in (r,\infty]$ and $q_\varphi\geq 1$ such that
\[
\frac{1}{\tilde{r}}+\frac{1}{\ell_r}+\frac{1}{q_\varphi}=1+\frac{1}{\sigma}.
\]
For example, we can choose $\tilde{r}=2r$, $q_\varphi=\sigma$. \par
Combining these, we get
\[
\left\|\int_{\mathbb{R}^N}\kappa_\varphi (x,y) |v_n(y)|^{p_{r;s}^{\uparrow *}}dy\right\|_{L^\sigma_x(\operatorname{supp}\varphi)} \leq C_4.
\]
Letting $\vartheta\coloneqq (\sigma-2r)/(2r+2r\sigma-\sigma)>0$, which satisfies
\[
\frac{1}{1+\vartheta}=\frac{1}{\sigma}+\frac{1}{\ell_r},
\]
by the H\"{o}lder's inequality, 
\begin{align*}
\int_{\operatorname{supp}\varphi}|\Gamma_n |^{1+\vartheta}dx &\leq \left\|\int_{\mathbb{R}^N}\kappa_\varphi (x,y) |v_n(y)|^{p_{r;s}^{\uparrow *}}dy\right\|_{L^\sigma_x(\operatorname{supp}\varphi)}^{1+\vartheta} \\
&\phantom{=}\quad \times \| |\varphi v_n|^{p_{r;s}^{\uparrow *}}\|_{\ell_r}^{1+\vartheta}\leq C_5.
\end{align*}
By the Vitali convergence theorem, from this and $\Gamma_n\to 0$ a.e. in $\mathbb{R}^N$, we obtain
\[
\int_{\operatorname{supp}\varphi}|\Gamma_n|dx\to 0.
\]
This, together with \eqref{Gamma_n}, and the Young's inequality for the weak Lebesgue spaces imply
\[
\int_{\mathbb{R}^N}(K\ast |v_n|^{p_{r;s}^{\uparrow *}}) |\varphi |^{2\cdot p_{r;s}^{\uparrow *}} |v_n|^{p_{r;s}^{\uparrow *}}dx \leq L_K^{-\frac{2\cdot p_{r;s}^{\uparrow *}}{p_s^*}}\|\varphi v_n\|_{p_s^*}^{2\cdot p_{r;s}^{\uparrow *}}+o(1).
\]
Taking the limit as $n$ goes to $\infty$, 
\[
\int_{\mathbb{R}^N}|\varphi |^{2\cdot p_{r;s}^{\uparrow *}} d\tilde{\xi} \leq L_K^{-\frac{2\cdot p_{r;s}^{\uparrow *}}{p_s^*}}\left(\int |\varphi|^{p_s^*}d\tilde{\nu}\right)^{2/\ell_r}.
\]
Applying Lemma I.2 in Part 1 of \cite{Lions}, we obtain the existence of $\{x_i\}$, $\{\xi_i\}$ satisfying the equality for $\xi$. Taking a sequence approximating $1_{\{x_i\}}$ as $\varphi$, we get $\nu_i\geq L_K\xi_i^{\ell_r /2}$ for $\nu_i=\nu(\{x_i\})$. On the other hand, taking a nondecreasing sequence of cut-off functions converging to $1_{\mathbb{R}^N}$ pointwise as $\varphi$, by the Lebesgue's convergence theorem with respect to fixed finite Borel measures $\nu$ and $\xi$, we also get $\nu(\mathbb{R}^N)\geq L_K\xi(\mathbb{R}^N)^{\ell_r /2}$. \par
Take any $\delta>0$. Then, there exists a constant $C(\delta)>2$ such that for any $p>1$ and for any $a,b\in\mathbb{R}$, we have $|a+b|^p\leq (1+\delta)|a|^p+C(\delta)|b|^p$. 
In the same way as in the case of $L_K$, by the definition of $S_K$, 
\begin{align*}
&\phantom{=}S_K \left(\int_{\mathbb{R}^N}|\varphi |^{2\cdot p_{r;s}^{\uparrow *}}(K\ast |v_n|^{p_{r;s}^{\uparrow *}}) |v_n|^{p_{r;s}^{\uparrow *}}dx\right)^{p/(2\cdot p_{r;s}^{\uparrow *})}\\
&\leq \|\varphi v_n\|_{D^{s,p}}^p+o_n(1) \\
&\leq (1+\delta)\int_{\mathbb{R}^N} \int_{\mathbb{R}^N}\frac{|\varphi(x)|^p|v_n(x)-v_n(y)|^p}{|x-y|^{N+sp}}dxdy \\
&\phantom{=}+C(\delta) \int_{\mathbb{R}^N} \int_{\mathbb{R}^N}\frac{|v_n(y)|^p|\varphi(x)-\varphi(y)|^p}{|x-y|^{N+sp}}dxdy+o_n(1).
\end{align*}
Let $\phi\in C_c^\infty(\mathbb{R}^N;[0,1])$ be such that $\phi=1$ in $B_1(0)$ while $\phi=0$ out of $B_2(0)$ and define $\phi_\varepsilon (x)\coloneqq \phi((x-x_i)/\varepsilon)$. Choose $\varphi=\phi_\varepsilon$. Arguing as in \cite{psLphierror}, \cite{Dspeta} or \cite{Dspphi}, we have
\[
\lim_{\varepsilon\to +0}\limsup_{n\to\infty}\int_{\mathbb{R}^N} \int_{\mathbb{R}^N}\frac{|v_n(y)|^p|\phi_\varepsilon(x)-\phi_\varepsilon(y)|^p}{|x-y|^{N+sp}}dxdy=0.
\]
Taking the limit as $n\to\infty$, we get
\[
S_K \left(\int_{\mathbb{R}^N}|\phi_\varepsilon |^{2\cdot p_{r;s}^{\uparrow *}} d\tilde{\xi}\right)^{p/(2\cdot p_{r;s}^{\uparrow *})}\leq (1+\delta)\int_{\mathbb{R}^N}| \phi_\varepsilon|^{p} d\tilde{\nu}+ C(\delta) o_\varepsilon (1).
\]
Using the arbitrariness of $\delta>0$ after taking the limit as $\varepsilon\to +0$, we obtain $S_K\xi_i^{p/(2\cdot p_{r;s}^{\uparrow *})}\leq \mu_i$, from which, together with the fact that the total variation of $\mu$ is finite and the monotonicity of $\ell^q$ quasi-norm in $q\in (0,\infty)$, we deduce
\[
\sum_{i\in \mathcal{I}}\xi_i^{p_s^*/(2 \cdot p_{r;s}^{\uparrow *})}\leq \sum_{i\in \mathcal{I}}\xi_i^{p/(2 \cdot p_{r;s}^{\uparrow *})} \leq S_K^{-1} \sum_{i\in \mathcal{I}}\mu_i \leq S_K^{-1}\mu(\mathbb{R}^N)<\infty.
\]
Analogously, choosing $\varphi(x)=\phi(x/R)$ and passing to the limit as $R\to\infty$, we obtain $S_K\xi (\mathbb{R}^N)^{p/(2\cdot p_{r;s}^{\uparrow *})}\leq \mu (\mathbb{R}^N)$. 
In the same way, using the Sobolev embedding, we can also obtain $S_{D^{s,p}}\nu_i^{p/p_s^*}\leq \mu_i$ and $S_{D^{s,p}}\nu (\mathbb{R}^N)^{p/p_s^*}\leq \mu (\mathbb{R}^N)$. \par
In order to consider the escaping parts, define $\eta_R (x)\coloneqq 1-\phi (x/R)$. For each of the integrands $f_n(x)= \displaystyle (K\ast |u_n|^{p_{r;s}^{\uparrow *}}) |u_n|^{p_{r;s}^{\uparrow *}}, \int_{\mathbb{R}^N}\frac{|u_n (x)-u_n (y)|^p}{|x-y|^{N+sp}}dy, |u_n|^{p_s^*}$, obviously we have
\begin{align*}
\limsup_{n\to\infty}\int_{\mathbb{R}^N}f_n dx &= \limsup_{n\to\infty}\int_{\mathbb{R}^N}f_n \eta_R dx+ \lim_{n\to\infty}\int_{\mathbb{R}^N} (1-\eta_R) (f_n dx).
\end{align*}
Consider the limit of this as $R\to\infty$. By applying the definition of $\xi_\infty$, $\mu_\infty$, $\nu_\infty$ to the first term and by the vague convergence as $n\to\infty$ and the Lebesgue's convergence theorem as $R\to\infty$ for the second term, we obtain \eqref{mass conservation for xi}, \eqref{mass conservation for mu}, \eqref{mass conservation for nu}. \par
By the Young's convolution inequality for the weak Lebesgue spaces, we have
\begin{align*}
\xi_\infty &= \lim_{R\to\infty}\limsup_{n\to\infty}\int_{\mathbb{R}^N} (K\ast |u_n|^{p_{r;s}^{\uparrow *}}) |\eta_R u_n |^{2\cdot p_{r;s}^{\uparrow *}}dx \\
&\leq L_K^{-2/\ell_r}\lim_{R\to\infty} \limsup_{n\to\infty}\left(\int_{\mathbb{R}^N} |u_n|^{p_s^*}dx\right)^{1/\ell_r}\left(\int_{\mathbb{R}^N} |\eta_R u_n|^{p_s^*}dx\right)^{1/\ell_r} \\
&= L_K^{-2/\ell_r}\left(\nu_\infty+\int_{\mathbb{R}^N}d\nu\right)^{1/\ell_r}\nu_\infty^{1/\ell_r},
\end{align*}
and so \eqref{xi and nu at infty} follows. Moreover, using the fractional Sobolev inequality as well as the Young's convolution inequality for the weak Lebesgue spaces, we have
\begin{align*}
\xi_\infty &\leq S_K^{-2\cdot p_{r;s}^{\uparrow *}/p} \lim_{R\to\infty} \limsup_{n\to\infty}\|u_n\|_{D^{s,p}}^{p_{r;s}^{\uparrow *}} \|\eta_R u_n\|_{D^{s,p}}^{p_{r;s}^{\uparrow *}} \\
&= S_K^{-2\cdot p_{r;s}^{\uparrow *}/p} \lim_{R\to\infty} \limsup_{n\to\infty}\left(\int_{\mathbb{R}^N} \int_{\mathbb{R}^N}\frac{|u_n(x)-u_n(y)|^p}{|x-y|^{N+sp}}\eta_R(x)dxdy \right. \\
&\phantom{=} \left. + \int_{\mathbb{R}^N} \int_{\mathbb{R}^N}\frac{|u_n(x)-u_n(y)|^p}{|x-y|^{N+sp}}(1-\eta_R(x))dxdy\right)^{p_{r;s}^{\uparrow *}/p} \times\|\eta_R u_n\|_{D^{s,p}}^{p_{r;s}^{\uparrow *}} \\
&= S_K^{-2\cdot p_{r;s}^{\uparrow *}/p} \left(\mu_\infty+ \int_{\mathbb{R}^N} d\mu\right)^{p_{r;s}^{\uparrow *}/p} \lim_{R\to\infty} \limsup_{n\to\infty} \|\eta_R u_n\|_{D^{s,p}}^{p_{r;s}^{\uparrow *}}
\end{align*}
For any $\delta>0$, there exists $C(\delta)>0$ such that
\begin{align*}
\|\eta_R u_n\|_{D^{s,p}}^{p} &\leq (1+\delta)\int_{\mathbb{R}^N} \int_{\mathbb{R}^N}\frac{|\eta_R (x)|^p|u_n(x)-u_n(y)|^p}{|x-y|^{N+sp}}dxdy \\
&\phantom{=}+C(\delta) \int_{\mathbb{R}^N} \int_{\mathbb{R}^N}\frac{|u_n(y)|^p|\eta_R (x)-\eta_R (y)|^p}{|x-y|^{N+sp}}dxdy.
\end{align*}
Noting that, in the same way as in \cite{Dspeta}, we have
\begin{align*}
&\phantom{=}\lim_{R\to\infty}\limsup_{n\to\infty}\int_{\mathbb{R}^N} \int_{\mathbb{R}^N}\frac{|u_n(y)|^p|\eta_R (x)-\eta_R (y)|^p}{|x-y|^{N+sp}}dxdy \\
& = \lim_{R\to\infty}\limsup_{n\to\infty}\int_{\mathbb{R}^N} \int_{\mathbb{R}^N}\frac{|(1-\eta_R (x))-(1-\eta_R (y))|^p}{|x-y|^{N+sp}} |u_n(y)|^p dxdy =0,
\end{align*}
we obtain
\[
\lim_{R\to\infty} \limsup_{n\to\infty} \|\eta_R u_n\|_{D^{s,p}}^{p}=\mu_\infty,
\]
and so \eqref{xi and mu at infty} follows. In the same way, using the Sobolev inequality, we can obtain \eqref{mugeqnuatinfty}.
\end{proof}

\section{$(PS)_c$ condition}
\begin{lemma}
There exists $\varepsilon_0>0$ such that if $\varepsilon_W\in (0,\varepsilon_0)$, then $I$ satisfies $(PS)_c$ condition for $c<0$.
\end{lemma}
\begin{proof}
Take any $(PS)_c$ sequence $\{u_n\}$ for $I$. Since $\{u_n\}$ is bounded and since $E$ is reflexive, up to a subsequence, there exists $u_0\in E$ such that $u_n\rightharpoonup u_0$ in $E$. By the compactness of embeddings, $u_n\to u_0$ in $L^q_\mathrm{loc}(\mathbb{R}^N)$ for any $q\in [1,p_s^*)$ and hence, up to a subsequence, $u_n\to u_0$ a.e. in $\mathbb{R}^N$. \par
Let $\mathcal{I}$, $\{x_i\}$, $\mu$, $\nu$, $\xi$, $\{\mu_i\}$, $\{\nu_i\}$, $\{\xi_i\}$, $\mu_\infty$, $\nu_\infty$, $\xi_\infty$ be as in Lemma \ref{CC}. Here, we note that $\mu$, $\nu$, $\xi$ actually exist due to Banach–Alaoglu theorem for $C_0(\mathbb{R}^N)'$. \par
First, we prove that for each $i\in \mathcal{I}$, either $\xi_i=0$ or $\xi_i\geq S_K^{\frac{N(r+1)}{N+prs}}$, which implies that $\mathcal{I}$ is finite because
\[
\sum_{i\in \mathcal{I}}\xi_i^{p_s^*/(2 \cdot p_{r;s}^{\uparrow *})}<\infty.
\]
Let $\varphi\in C_c^\infty (\mathbb{R}^N)$ be a cut-off function such that $0\leq \varphi\leq 1$, $\varphi(x)=0$ for $|x|\geq 2$, $\varphi(x)=1$ for $|x|\leq 1$ and define $\varphi_{a,\varepsilon}((x-a)/\varepsilon)$. For each $\varepsilon>0$ and each $i\in \mathcal{I}$, since $\{\varphi_{x_i,\varepsilon} u_n\}$ is bounded in $E$, we have $I'[u_n](\varphi_{x_i,\varepsilon} u_n)\to 0$ as $n\to\infty$. \par
As for $p$-fractional part, 
\begin{align*}
&\int_{\mathbb{R}^{2N}}\frac{|u_n(x)-u_n(y)|^{p-2}(u_n(x)-u_n(y))(\varphi_{x_i,\varepsilon}(x)u_n(x)-\varphi_{x_i,\varepsilon}(y)u_n(y))}{|x-y|^{N+ps}}dxdy \\
&= \int_{\mathbb{R}^{2N}}\frac{|u_n(x)-u_n(y)|^{p}\varphi_{x_i,\varepsilon}(x)}{|x-y|^{N+ps}}dxdy \\
&\phantom{=}+ \int_{\mathbb{R}^{2N}}\frac{|u_n(x)-u_n(y)|^{p-2}(u_n(x)-u_n(y))(\varphi_{x_i,\varepsilon}(x)-\varphi_{x_i,\varepsilon}(y))u_n(y)}{|x-y|^{N+ps}}dxdy \\
&\eqqcolon I_{1,n}'+I_{2,n}'.
\end{align*}
For $\varepsilon>0$ sufficiently small, 
\begin{align*}
&\phantom{=}\lim_{n\to\infty} I_{1,n}'=\int_{\mathbb{R}^N} \varphi_{x_i,\varepsilon}d\mu 
=  \int_{\mathbb{R}^N} \int_{\mathbb{R}^N}\frac{|u_0(x)-u_0(y)|^{p}\varphi_{x_i,\varepsilon}(x)}{|x-y|^{N+ps}}dxdy +\mu_i,
\end{align*}
which converges to $\mu_i$ as $\varepsilon\to +0$. \par
On the other hand, by the H\"{o}lder's inequality for weighted Lebesgue spaces on $\mathbb{R}^{2N}$, we have
\begin{align*} 
|I_{2,n}'|&\leq \left(\int_{\mathbb{R}^N} \int_{\mathbb{R}^N}\frac{|u_n(x)-u_n(y)|^{p}}{|x-y|^{N+ps}}dxdy\right)^{\frac{p-1}{p}} \\
&\phantom{=}\times\left(\int_{\mathbb{R}^N} \int_{\mathbb{R}^N}\frac{|\varphi_{x_i,\varepsilon}(x)-\varphi_{x_i,\varepsilon}(y)|^{p}|u_n(y)|^p}{|x-y|^{N+ps}}dxdy\right)^{\frac{1}{p}} \\
&\leq \left(\sup_n\|u_n\|_{D^{s,p}}\right)^{p-1} \left(\int_{\mathbb{R}^N} \int_{\mathbb{R}^N}\frac{|\varphi_{x_i,\varepsilon}(x)-\varphi_{x_i,\varepsilon}(y)|^{p}|u_n(y)|^p}{|x-y|^{N+ps}}dxdy\right)^{\frac{1}{p}}.
\end{align*}
By Lemma 2.3 in \cite{psLphierror}, the right hand side vanishes after taking limits:
\[
\lim_{\varepsilon\to +0}\limsup_{n\to\infty}
\int_{\mathbb{R}^N} \int_{\mathbb{R}^N}\frac{|\varphi_{x_i,\varepsilon}(x)-\varphi_{x_i,\varepsilon}(y)|^{p}|u_n(y)|^p}{|x-y|^{N+ps}}dxdy=0.
\]
Furthermore, 
\begin{align*}
&\phantom{=}\lim_{n\to\infty} \int_{\mathbb{R}^N}(K\ast |u_n|^{p_{r;s}^{\uparrow *}}) |u_n|^{p_{r;s}^{\uparrow *}} \varphi_{x_i,\varepsilon}dx \\
&= \int_{\mathbb{R}^N} \varphi_{x_i,\varepsilon}d\xi \\
&= \lim_{n\to\infty} \int_{\mathbb{R}^N}(K\ast |u_0|^{p_{r;s}^{\uparrow *}}) |u_0|^{p_{r;s}^{\uparrow *}} \varphi_{x_i,\varepsilon}dx +\xi_i \\
&\to \xi_i\quad (\varepsilon\to +0).
\end{align*}
As for subcritical parts, by usual H\"{o}lder's inequality and the Young's inequality for weak Lebesgue spaces, as $\varepsilon\to +0$, 
\begin{align*}
\limsup_{n\to\infty} \left|\int_{\mathbb{R}^N}V|u_n|^p \varphi_{x_i,\varepsilon} dx\right| &\leq \|V\|_{L^{N/(ps)}(B_{2\varepsilon}(x_i))}\sup_n\|u_n\|_{p_s^*}^p \\
&\to 0\quad (\varepsilon\to +0);\\
\limsup_{n\to\infty} \int_{\mathbb{R}^N}(K\ast g(u_n))h'(u_n)u_n \varphi_{x_i,\varepsilon}dx &\to 0 \quad (\varepsilon\to +0);\\
\limsup_{n\to\infty} \int_{\mathbb{R}^N}(K\ast h(u_n))|u_n|^{p_{r;s}^{\uparrow *}} \varphi_{x_i,\varepsilon} dx &\to 0 \quad (\varepsilon\to +0);\\
\limsup_{n\to\infty} \int_{\mathbb{R}^N}W(x)f'(u_n)u_n \varphi_{x_i,\varepsilon} dx &\leq C\sum_{j=1,2}\|W\|_{L^{\frac{p_s^*}{p_s^*-q_j}}(B_{2\varepsilon}(x_i))}\sup_n\|u_n\|_{p_s^*}^{q_j} \\
&\to 0 \quad (\varepsilon\to +0).
\end{align*}
Therefore, from the termwise estimation for $I'[u_n](\varphi_{x_i,\varepsilon} u_n)$, we obtain 
\[
0\geq \mu_i-\xi_i.
\]
Combining this with \eqref{CCquantity}, for each $i$, either $\xi_i=0$ or $\xi_i\geq S_K^{\frac{1}{1-\frac{p}{2\cdot p_{r;s}^{\uparrow *}}}}=S_K^{\frac{N(r+1)}{N+prs}}$. \par

In order to rule out the possibility of $\xi_i\geq S_K^{\frac{N(r+1)}{N+prs}}$ (and to rule out the possibility of mass escaping to the infinity as mentioned later), we observe that by the weak lower semicontinuity of $\|\cdot\|_E$ and the weak continuity of $u\mapsto \int_{\mathbb{R}^N}W(x)f'(u)u dx$, for $\beta\in (1/(2\alpha_g),1/p)\subset (1/(2\alpha_g),1/\alpha_f)$, we have
\begin{align*}
0>c&=\lim_{n\to\infty}\left(I[u_n]-\beta I'[u_n]u_n\right) \\
&\geq \left(\frac{1}{p}-\beta\right)\|u_n\|_E^p+\frac{2\beta-1/\alpha_g}{2 \cdot p_{r;s}^{\uparrow *}}\int_{\mathbb{R}^N}(K\ast g(u_n))g'(u_n)u_n dx \\
&\phantom{=}-\varepsilon_W \left(\frac{1}{\alpha_f}-\beta\right)\int_{\mathbb{R}^N}W(x)f'(u_n)u_n dx \\
&\geq \left(\frac{1}{p}-\beta\right)\|u_0\|_E^p-\varepsilon_W \left(\frac{1}{\alpha_f}-\beta\right)\int_{\mathbb{R}^N}W(x)f'(u_0)u_0 dx \\
&\geq \left(\frac{1}{p}-\beta\right)\|u_0\|_E^p-\varepsilon_W \cdot C_1\sum_{j=1,2}\|W\|_{\frac{p_s^*}{p_s^*-q_j}}\|u_0\|_{p_s^*}^{q_j} \\
&\geq \left(\frac{1}{p}-\beta\right)S_{D^{s,p}}\|u_0\|_{p_s^*}^p-\varepsilon_W \cdot C_2(\|u_0\|_{p_s^*}^{q_1}+ \|u_0\|_{p_s^*}^{q_2}).
\end{align*}
Since $\ell(t)\coloneqq t^p/(t^{q_1}+t^{q_2})$ is strictly monotonically increasing, we can write $\|u_0\|_{p_s^*} \leq \ell^{-1}(C_3 \varepsilon_W)$. \par

In order to analyze the concentration at the infinity, we define $\eta_R (x)\coloneqq 1-\varphi (x/R)$. Applying $I'[u_n]$ to the test function $u_n \eta_R$ and taking limits, we get
\begin{equation}\label{inftyquantity}
\mu_\infty+V_\infty=\xi_\infty+\xi_\infty'+\varepsilon_W W_\infty
\end{equation}
where
\begin{align*}
V_\infty &\coloneqq \lim_{R\to\infty}\limsup_{n\to\infty} \int_{\mathbb{R}^N}V|u_n|^p \eta_R dx; \\
\xi_\infty' &\coloneqq \lim_{R\to\infty}\limsup_{n\to\infty}\left(\frac{1}{p_{r;s}^{\uparrow *}}\int_{\mathbb{R}^N} (K\ast g(u_n))g'(u_n)u_n\eta_R dx\right)-\xi_\infty \\
&= \lim_{R\to\infty}\limsup_{n\to\infty}\left(\frac{1}{p_{r;s}^{\uparrow *}}\int_{\mathbb{R}^N} (K\ast g(u_n))h'(u_n)u_n\eta_R dx \right. \\
&\phantom{=}\left. +\int_{\mathbb{R}^N} (K\ast g(u_n))|u_n|^{p_{r;s}^{\uparrow *}}\eta_R dx\right); \\
W_\infty &\coloneqq \lim_{R\to\infty}\limsup_{n\to\infty} \int_{\mathbb{R}^N}W(x)f'(u_n)u_n \eta_R dx.
\end{align*}
However, $W_\infty=0$. Indeed, 
\begin{align*}
\int_{\mathbb{R}^N}W(x)|u_n|^{q_j}\eta_R dx &\leq \left(\int_{\{|x|\geq R\}}|u_n|^{p_s^*}dx\right)^{q_j/p_s^*} \left(\int_{\{|x|\geq R\}}|W|^{\frac{p_s^*}{p_s^*-q_j}}dx\right)^{1-\frac{q_j}{p_s^*}} \\
&\leq S_{D^{s,p}}^{-q_j/p_s^*}\|u_n\|_{D^{s,p}}^{q_j} \left(\int_{\{|x|\geq R\}}|W|^{\frac{p_s^*}{p_s^*-q_j}}dx\right)^{1-\frac{q_j}{p_s^*}}
\end{align*}
and $\{u_n\}$ is bounded in $D^{s,p}$. Combining this with \eqref{inftyquantity}, also we have
\begin{equation}\label{contradictionforsmallW}
\begin{split}
0>c&=\lim_{n\to\infty}\left(I[u_n]-\beta I'[u_n]u_n\right) \\
&\geq \frac{2\beta-1/\alpha_g}{2 \cdot p_{r;s}^{\uparrow *}}\int_{\mathbb{R}^N}(K\ast g(u_n))g'(u_n)u_n dx -\varepsilon_W \left(\frac{1}{\alpha_f}-\beta\right)\int_{\mathbb{R}^N}W(x)f'(u_n)u_n dx \\
&\geq \left(\beta-\frac{1}{2\alpha_g}\right)(\int_{\mathbb{R}^N}d\xi+\xi_\infty+\xi_\infty') -\varepsilon_W \left(\frac{1}{\alpha_f}-\beta\right)\int_{\mathbb{R}^N}W(x)f'(u_0)u_0 dx+o_n (1) \\
&\geq \left(\beta-\frac{1}{2\alpha_g}\right)(\sum_{i\in\mathcal{I}}\xi_i+\xi_\infty+\xi_\infty') -\varepsilon_W \left(\frac{1}{\alpha_f}-\beta\right)\int_{\mathbb{R}^N}W(x)f'(u_0)u_0 dx+o_n (1) \\
&= \left(\beta-\frac{1}{2\alpha_g}\right)(\sum_{i\in\mathcal{I}}\xi_i+\mu_\infty+V_\infty) -\varepsilon_W \left(\frac{1}{\alpha_f}-\beta\right)\int_{\mathbb{R}^N}W(x)f'(u_0)u_0 dx+o_n (1) \\
&\geq \left(\beta-\frac{1}{2\alpha_g}\right)(\sum_{i\in\mathcal{I}}\xi_i+\mu_\infty) -\varepsilon_W \left(\frac{1}{\alpha_f}-\beta\right)\int_{\mathbb{R}^N}W(x)f'(u_0)u_0 dx+o_n (1) \\
&\geq \left(\beta-\frac{1}{2\alpha_g}\right)(\sum_{i\in\mathcal{I}}\xi_i+\mu_\infty) -\varepsilon_W \cdot C_2(\|u_0\|_{p_s^*}^{q_1}+ \|u_0\|_{p_s^*}^{q_2}) +o_n (1) \\
&\geq \left(\beta-\frac{1}{2\alpha_g}\right)(\sum_{i\in\mathcal{I}}\xi_i+\mu_\infty) - C_2 \varepsilon_W(\ell^{-1}(C_3 \varepsilon_W)^{q_1}+ \ell^{-1}(C_3 \varepsilon_W)^{q_2}) +o_n (1).
\end{split}
\end{equation}
Noting that $\ell(t)\sim t^{p-\min\{q_1,q_2\}}$ as $t\to +0$, we can observe \[
t (\ell^{-1}(t)^{q_1}+ \ell^{-1}(t)^{q_2})\sim t^{1+\frac{\min\{q_1,q_2\}}{p-\min\{q_1,q_2\}}}= t^{\frac{p}{p-\min\{q_1,q_2\}}}
\]
as $t\to +0$. \par
Suppose $\xi_i\geq S_K^{\frac{N(r+1)}{N+prs}}$ for some $i\in\mathcal{I}$. Then, 
\begin{align*}
S_K^{\frac{N(r+1)}{N+prs}} &< C_2 \varepsilon_W (\ell^{-1}(C_3 \varepsilon_W)^{q_1}+ \ell^{-1}(C_3 \varepsilon_W)^{q_2})= O(\varepsilon_W^{\frac{p}{p-\min\{q_1,q_2\}}})
\end{align*}
Therefore, if $\varepsilon_W>0$ is sufficiently small, we reach a contradiction. Hence, $\xi_i=0$ for any $i\in\mathcal{I}$. \par
Next, we consider the escaping parts $\mu_\infty,\nu_\infty,\xi_\infty,\xi_\infty',V_\infty$ and prove these all are equal to zero. \par
By the H\"{o}lder's inequality, we have
\begin{align*}
\int_{\mathbb{R}^N}V|u_n|^p\eta_R dx&= \int_{\mathbb{R}^N}((V-\tau_0)_{+}+\tau_0)|u_n|^p\eta_R dx-\int_{\mathbb{R}^N}(V-\tau_0)_{-}|u_n|^p\eta_R dx \\
&\geq \int_{\mathbb{R}^N}\tau_0|u_n|^p\eta_R dx-\|(V-\tau_0)_{-}\|_{L^{\frac{N}{ps}}(\mathbb{R}^N\setminus B_R(0))} \|u_n\|_{p_s^*}^{p}
\end{align*}
and thus
\[
V_\infty=\lim_{R\to\infty}\limsup_{n\to\infty} \int_{\mathbb{R}^N}V|u_n|^p\eta_R dx\geq \tau_0\lim_{R\to\infty}\limsup_{n\to\infty} \int_{\mathbb{R}^N} |u_n|^p\eta_R dx.
\]
Take a number $\alpha'\in (0,1)$ such that 
\[
\frac{1}{\hat{p}_g\ell_r}=\frac{\alpha'}{p}+\frac{1-\alpha'}{p_s^*},
\]
that is, explicitly
\[
\alpha'=\frac{p(p_s^*-\hat{p}_g\ell_r)}{\hat{p}_g\ell_r(p_s^*-p)}.
\]
By the Young's inequality for the weak Lebesgue spaces and the H\"{o}lder's inequality, we have
\begin{align*}
&\phantom{=}\lim_{R\to\infty}\limsup_{n\to\infty} \int_{\mathbb{R}^N} (K\ast g(u_n))g'(u_n)u_n\eta_R dx \\
&\leq C_4 \lim_{R\to\infty}\limsup_{n\to\infty} \int_{\mathbb{R}^N} (K\ast (|u_n|^{\hat{p}_g}+|u_n|^{p_{r;s}^{\uparrow *}})) (|u_n|^{\hat{p}_g}+|u_n|^{p_{r;s}^{\uparrow *}})\eta_R dx \\
&= C_4 \lim_{R\to\infty}\limsup_{n\to\infty} \int_{\mathbb{R}^N} (K\ast |u_n|^{\hat{p}_g}) |u_n|^{\hat{p}_g} \eta_R dx\\
&\phantom{=} +C_4 \lim_{R\to\infty}\limsup_{n\to\infty} \int_{\mathbb{R}^N} (K\ast |u_n|^{\hat{p}_g}) |u_n|^{p_{r;s}^{\uparrow *}} \eta_R dx\\
&\phantom{=} +C_4 \lim_{R\to\infty}\limsup_{n\to\infty} \int_{\mathbb{R}^N} (K\ast |u_n|^{p_{r;s}^{\uparrow *}}) |u_n|^{\hat{p}_g} \eta_R dx\\
&\phantom{=} +C_4 \lim_{R\to\infty}\limsup_{n\to\infty} \int_{\mathbb{R}^N} (K\ast |u_n|^{p_{r;s}^{\uparrow *}}) |u_n|^{p_{r;s}^{\uparrow *}} \eta_R dx\\
&\leq C_5 \lim_{R\to\infty}\limsup_{n\to\infty}\|u_n\|_{\hat{p}_g\ell_r}^{\hat{p}_g} \| |u_n|^{\hat{p}_g}\eta_R \|_{\ell_r} \\
&\phantom{=} +C_6 \lim_{R\to\infty}\limsup_{n\to\infty} \|u_n\|_{\hat{p}_g\ell_r}^{\hat{p}_g} \| |u_n|^{p_{r;s}^{\uparrow *}}\eta_R \|_{\ell_r}\\
&\phantom{=} +C_7 \lim_{R\to\infty}\limsup_{n\to\infty} \|u_n\|_{p_{r;s}^{\uparrow *}\ell_r}^{p_{r;s}^{\uparrow *}}\| |u_n|^{\hat{p}_g}\eta_R \|_{\ell_r}\\
&\phantom{=} +C_8 \lim_{R\to\infty}\limsup_{n\to\infty} \|u_n\|_{p_{r;s}^{\uparrow *}\ell_r}^{p_{r;s}^{\uparrow *}}\| |u_n|^{p_{r;s}^{\uparrow *}}\eta_R \|_{\ell_r}\\
&\leq C_5 (\sup_n\|u_n\|_{\hat{p}_g\ell_r}^{\hat{p}_g})\lim_{R\to\infty}\limsup_{n\to\infty} \left(\int_{\mathbb{R}^N} |u_n|^{p}\eta_R dx\right)^{\frac{\alpha'\hat{p}_g}{p}} \left(\int_{\mathbb{R}^N} |u_n|^{p_s^*}\eta_R dx\right)^{\frac{(1-\alpha')\hat{p}_g}{p_s^*}}\\
&\phantom{=} +C_6 (\sup_n \|u_n\|_{\hat{p}_g\ell_r}^{\hat{p}_g}) \lim_{R\to\infty}\limsup_{n\to\infty}\left(\int_{\mathbb{R}^N} |u_n|^{p_s^*}\eta_R dx\right)^{1/\ell_r}\\
&\phantom{=} +C_7 (\sup_n \|u_n\|_{p_s^*}^{p_{r;s}^{\uparrow *}}) \lim_{R\to\infty}\limsup_{n\to\infty} \left(\int_{\mathbb{R}^N} |u_n|^{p}\eta_R dx\right)^{\frac{\alpha'\hat{p}_g}{p}} \left(\int_{\mathbb{R}^N} |u_n|^{p_s^*}\eta_R dx\right)^{\frac{(1-\alpha')\hat{p}_g}{p_s^*}}\\
&\phantom{=} +C_8 (\sup_n\|u_n\|_{p_s^*}^{p_{r;s}^{\uparrow *}}) \lim_{R\to\infty}\limsup_{n\to\infty} \left(\int_{\mathbb{R}^N} |u_n|^{p_s^*}\eta_R dx\right)^{1/\ell_r}\\
&\leq C_9 \left(\frac{V_\infty}{\tau_0}\right)^{\frac{\alpha'\hat{p}_g}{p}} \left(\nu_\infty\right)^{\frac{(1-\alpha')\hat{p}_g}{p_s^*}} +C_{10}\left(\nu_\infty\right)^{1/\ell_r}\\
&\phantom{=} +C_{11} \left(\frac{V_\infty}{\tau_0}\right)^{\frac{\alpha'\hat{p}_g}{p}} \left(\nu_\infty\right)^{\frac{(1-\alpha')\hat{p}_g}{p_s^*}} +C_{12} \left(\nu_\infty\right)^{1/\ell_r}.
\end{align*}
where constants $C_4,C_5,\ldots,C_{12}$ depend only on best constants for embeddings and $K,V,W,f,g$ (see Remark \ref{uniformboundforW}). Therefore, from \eqref{inftyquantity}, we know
\begin{equation}\label{nazonoquantity}
\mu_\infty+V_\infty\leq C\left(\frac{V_\infty}{\tau_0}\right)^{\frac{\alpha'\hat{p}_g}{p}} \left(\nu_\infty\right)^{\frac{(1-\alpha')\hat{p}_g}{p_s^*}}+C'\nu_\infty^{1/\ell_r}.
\end{equation}
Suppose $\nu_\infty>0$. 
In the same way as in \cite{6}, the Young's inequality for products 
yields the existence of $\Lambda_0>0$ such that $\nu_\infty>\Lambda_0$ depending only on embedding constants and $K,V,W,f,g$. This also implies $\mu_\infty>S_{D^{s,p}}\Lambda_0^{p/p_s^*}$. If $\varepsilon_W>0$ is sufficiently small, this lead to a contradiction together with \eqref{contradictionforsmallW}. Therefore, we can deduce $\nu_\infty=0$ and thus all the escaping parts vanish by using \eqref{CCquantity}, \eqref{inftyquantity} and \eqref{nazonoquantity}. \par
Especially, from $\xi_i=\xi_\infty=\xi_\infty'=0$ for any $i\in\mathcal{I}$, we now know $J[u_n]\to J[u_0]$ and $J'[u_n]u_n\to J'[u_0]u_0$; that is, 
\begin{equation}
\lim_{n\to\infty}\int_{\mathbb{R}^N} (K\ast g(u_n))g(u_n)dx= \int_{\mathbb{R}^N} (K\ast g(u_0))g(u_0)dx
\end{equation}
and
\begin{equation}\label{Jprimeunun}
\lim_{n\to\infty}\int_{\mathbb{R}^N} (K\ast g(u_n))g'(u_n)u_n dx= \int_{\mathbb{R}^N} (K\ast g(u_0))g'(u_0)u_0 dx.
\end{equation}
Moreover, from $V_\infty=0$, we know
\[
\lim_{n\to\infty}\int_{\mathbb{R}^N} V|u_n|^p= \int_{\mathbb{R}^N} V|u_0|^p dx.
\]
By the Brezis-Lieb splitting, this implies
\[
\lim_{n\to\infty}\int_{\mathbb{R}^N} V|u_n-u_0|^p dx=0.
\]
Now it suffices to check $\|u_n-u_0\|_{D^{s,p}}\to 0$ as $n\to\infty$. \par
Since
\begin{align*}
&\phantom{=}\sup_n\int_{\Omega}|W(x)f'(u_n)u_n|dx \\
&\leq \sum_{j=1,2}\|W\|_{L^{\frac{p_s^*}{p_s^*-q_j}}(\Omega)}\sup_n(\|u_n\|_{p_s^*}^{q_j-1}\|u_n\|_{p_s^*}) \\
&\leq C\sum_{j=1,2}\|W\|_{L^{\frac{p_s^*}{p_s^*-q_j}}(\Omega)}
\end{align*}
for any domain $\Omega\subset\mathbb{R}^N$, the sequence $\{W(x)f'(u_n)u_n\}$ is equi-integrable (in the sense that it is tight and has uniformly absolutely continuous integrals). By the Vitali convergence theorem for infinite measure spaces, we get
\[
\lim_{n\to\infty}\int_{\mathbb{R}^N}W(x)f'(u_n)u_n dx= \int_{\mathbb{R}^N}W(x)f'(u_0)u_0 dx.
\]
Therefore, comparing $I'[u_n]u_n=o(1)$ and $I'[u_0]u_0=0$, we get $\|u_n\|_{D^{s,p}}^p= \|u_0\|_{D^{s,p}}^p$. Hence, together with the fact that $u_n\rightharpoonup u_0$ weakly in $D^{s,p}(\mathbb{R}^N)$, we can deduce $u_n\to u_0$ in $D^{s,p}(\mathbb{R}^N)$. 
\end{proof}

\section{Proof of the main theorem}
\begin{definition}
Let $X$ be a Banach space and $\Sigma$ be the class of all the closed subsets of $X\setminus\{0\}$ which are symmetric (with respect to the origin of $X$). For $A\in\Sigma$, we define the genus $\gamma(A)$ of $A$ by
\[
\gamma(A)\coloneqq \inf\{n\in\mathbb{N}\mid \exists \varphi\in C(A,\mathbb{R}^n)\text{ s.t. }\varphi(-u)=-\varphi(u)\}.
\]
Furthermore, we define
\[
\Sigma_n\coloneqq \{A\in\Sigma\mid \gamma(A)\geq n\}.
\]
\end{definition}
As for the properties of genus, see \cite{CriticalValue}. We utilize the following version of the symmetric mountain pass lemma due to Ambrosetti-Rabinowitz \cite{SymmetricMountainPass}.
\begin{proposition}\label{symmetric mountain pass}
Let $X$ be an infinite-dinensional Banach space and suppose that even functional $I\in C^1(X;\mathbb{R})$ bounded from below with $I(0)=0$ satisfies the $(PS)_c$ condition for $c<0$. Assume that for each $n\in\mathbb{N}$, there exists $A_n\in \Sigma_n$ such that $\displaystyle \sup_{A_n}I<0$. Then, each $\displaystyle c_n\coloneqq \inf_{A\in\Sigma_n}\sup_A I$ is a critical value of $I$ and $c_n\to 0$ ($n\to\infty$). \par
In particular, there exists a sequence of critical points $u_n\neq 0$ such that $I[u_n]\leq 0$, $u_n\to 0$ in $X$.
\end{proposition}
Moreover, Kajikiya \cite{Kajikiya} classified the possible behavior of the sequences of critical points under such situations.
\begin{proposition}
Under the assumptions of Proposition \ref{symmetric mountain pass}, one of the followings holds:
\begin{enumerate}
\item There exists a sequence $\{u_n\}\subset X$ such that $I'[u_n]=0$, $I[u_n]<0$, $u_n\to 0$ in $X$.
\item There exist two sequence $\{u_n\}$ and $\{v_n\}$ such that $I'[u_n]=0$, $I[u_n]=0$, $u_n\neq 0$, $u_n\to 0$ in $X$ and $I'[v_n]=0$, $I[v_n]<0$, $v_n\to {}^\exists v\neq 0$ in $X$, $I[v]=0$.
\end{enumerate}
\end{proposition}
Here, we estimate $I$ from below. 
\begin{align*}
I[u]&\geq C_1\|u\|_{E}^p-\frac{1}{2\cdot p_{r;s}^{\uparrow *}}\int_{\mathbb{R}^N}(K\ast g(u))g(u)dx-\varepsilon_W\int_{\mathbb{R}^N} W(x)f(u)dx \\
&\geq C_1\|u\|_{E}^p-C_2\int_{\mathbb{R}^N}(K\ast (|u|^{\hat{p}_g}+|u|^{p_{r;s}^{\uparrow *}})) (|u|^{\hat{p}_g}+|u|^{p_{r;s}^{\uparrow *}})dx \\
&\phantom{=} -\varepsilon_W \cdot C_3\sum_{j=1,2}\|W\|_{\frac{p_s^*}{p_s^*-q_j}}\|u\|_{p_s^*}^{q_j} \\
&\geq C_1\|u\|_{E}^p-C_2 \|K\|_{L^{r,\infty}}(\|u^{\hat{p}_g}\|_{\ell_r}^2+2 \|u^{\hat{p}_g}\|_{\ell_r} \|u^{p_{r;s}^{\uparrow *}}\|_{\ell_r}+ \|u^{p_{r;s}^{\uparrow *}}\|_{\ell_r}^{2}) \\
&\phantom{=} -\varepsilon_W \cdot C_3'(\|u\|_E^{q_1}+ \|u\|_E^{q_2}) \\
&\geq C_1\|u\|_{E}^p-C_2' (\|u\|_{E}^{2\hat{p}_g}+\|u\|_{E}^{2\cdot p_{r;s}^{\uparrow *}}) -\varepsilon_W \cdot C_3'(\|u\|_E^{q_1}+ \|u\|_E^{q_2}).
\end{align*}
Let $\ell_\varepsilon (t)\coloneqq C_1 t^p-C_2' (t^{2\hat{p}_g}+t^{2\cdot p_{r;s}^{\uparrow *}})-C_3'\varepsilon(t^{q_1}+t^{q_2})$. 
Since $\max\{q_1,q_2\}<p<2\hat{p}_g\leq 2\cdot p_{r;s}^{\uparrow *}$, there exists $\varepsilon_0'>0$ so small that for any $\varepsilon\in (0,\varepsilon_0')$, there exist $t_{\varepsilon,0}$ and $t_{\varepsilon,1}$ with $0<t_{\varepsilon,0}<t_{\varepsilon,1}$ such that $\ell_\varepsilon (t)<0$ for $t\in (0, t_{\varepsilon,0})$, $\ell_\varepsilon (t)>0$ for $t\in (t_{\varepsilon,0}, t_{\varepsilon,1})$, $\ell_\varepsilon (t) <0$ for $t> t_{\varepsilon,1}$. To see this, it suffices to consider some small perturbation of the graph of $t\mapsto C_1 t^p-C_2' (t^{2\hat{p}_g}+t^{2\cdot p_{r;s}^{\uparrow *}})$ near the origin. \par
Here, take $\psi\in C^\infty(\mathbb{R})$ such that $\psi(s)=1$ for $s\in [0, t_{\varepsilon_W,0}^p)$ and $\psi(s)=0$ for $s> t_{\varepsilon_W,1}^p$. Let us define the truncated functional
\begin{align*}
\tilde{I}[u] &\coloneqq I_{+}[u] -\psi(\|u\|_E^p) I_{-}[u] \\
&=\frac{1}{p}\|u\|_{D^{s,p}}^p+ \frac{1}{p}\int_{\mathbb{R}^N}V(x)|u|^p dx \\
&\phantom{=} -\frac{1}{2\cdot p_{r;s}^{\uparrow *}} \psi(\|u\|_E^p)\int_{\mathbb{R}^N}(K\ast g(u))g(u)dx-\varepsilon_W \psi(\|u\|_E^p)\int_{\mathbb{R}^N}W(x)f(u)dx.
\end{align*}
Then, $\tilde{I}$ is even, bounded from below and satisfies $(PS)_c$ condition for any $c<0$ and for any $\varepsilon_W\in (0,\min\{\varepsilon_0,\varepsilon_0'\})$. Let us also note that if $\tilde{I}[u]<0$, from $\psi(\|u\|_E)>0$, we get $\|u\|_E< t_{\varepsilon_W,1}$ and moreover, from $\ell_{\varepsilon_W}(\|u\|_E)<0$, we get $\|u\|_E< t_{\varepsilon_W,0}$ and thus $\tilde{I}[u]=I[u]$. \par
Now we are in a position to check the assumptions of the symmetric mountain pass lemma for $\tilde{I}$. 
\begin{lemma}
Assume $0<\varepsilon_W<\varepsilon_0'$. For each $n\in\mathbb{N}$, there exists $\delta_n<0$ such that $\gamma(\tilde{I}^{-1}((-\infty,\delta_n]))\geq n$. 
\end{lemma}
\begin{proof}
Take a $n$-dimensional subspace $X_n$ of $E$. For each $u\in X_n\setminus\{0\}$, write $u=rv$ with $r\coloneqq \|u\|$. Since $X_n\cap \{\|u\|=1\}$ is compact and $W,K>0$, by the Ambrosetti-Rabinowitz type condition for $f$, there exist $d_n>0$ and $e_n>0$ such that 
\[
\int_{\mathbb{R}^N}W(x)|v|^{\alpha_f}dx\geq d_n
\]
and
\[
\int_{\mathbb{R}^N}(K\ast |v|^{p_{r;s}^{\uparrow *}})|v|^{p_{r;s}^{\uparrow *}}dx\geq e_n
\]
for any $v\in X_n$ with $\|v\|=1$. For $0<r<t_{\varepsilon_W,0}$, we can estimate $\tilde{I}$ as follows. 
\begin{align*}
\tilde{I}[u]&\leq\frac{1}{p}\|u\|^p-\frac{1}{2\cdot p_{r;s}^{\uparrow *}} \int_{\mathbb{R}^N}(K\ast |u|^{p_{r;s}^{\uparrow *}})|u|^{p_{r;s}^{\uparrow *}}dx -C_0 \varepsilon_W \int_{\mathbb{R}^N}W(x)|u|^{\alpha_f}dx \\
&\leq C_1 r^p-C_2 e_n r^{2\cdot p_{r;s}^{\uparrow *}}-C_3\varepsilon_W d_n r^{\alpha_f} \eqqcolon \delta_n
\end{align*}
Since $\alpha_f<p$, there exists $r\in (0,t_{\varepsilon_W,0})$ so small that the right hand side is strictly negative: $\delta_n<0$. Now we have $X_n\cap\{\|u\|=r\}\subset \tilde{I}^{-1}((-\infty,\delta_n])$. Consequently,
\[
\gamma(\tilde{I}^{-1}((-\infty,\delta_n]))\geq \gamma(X_n\cap\{\|u\|=r\})=\gamma (S^{n-1})=n.
\]
\end{proof}
From this lemma, since $\tilde{I}^{-1}((-\infty,\delta_n]) \in\Sigma_n$, we know
\[
c_n\coloneqq\inf_{A\in\Sigma_n}\sup_{A} \tilde{I}\leq\delta_n<0.
\]
Moreover, $\{c_n\}$ is a sequence of critical values of $\tilde{I}$ (see \cite{CriticalValue}). \par
Now we prove Theorem \ref{MainTheorem}.
\begin{proof}
Since $\tilde{I}$ is bounded from below, we have $c_n>-\infty$. From $\Sigma_{n+1}\subset\Sigma_n$, it follows that $c_{n+1}\geq c_n$. Therefore, there exists $\bar{c}\leq 0$ such that $c_n\to \bar{c}$. Suppose $\bar{c}<0$ for the sake of contradiction. By the $(PS)_{\bar{c}}$ condition, the set $\mathcal{C}_{\tilde{I}, \bar{c}}\coloneqq \{u\in E\mid \tilde{I}'[u]=0, \tilde{I}[u]= \bar{c}\}\in\Sigma$ is compact. Hence, $\gamma(\mathcal{C}_{\tilde{I}, \bar{c}})<\infty$ and there exists $\delta>0$ such that $N_\delta (\mathcal{C}_{\tilde{I}, \bar{c}})\in\Sigma$ and $\gamma(N_\delta (\mathcal{C}_{\tilde{I}, \bar{c}}))=\gamma(\mathcal{C}_{\tilde{I}, \bar{c}})$ where $N_\delta$ denotes $\delta$-neighborhood of the subset of $E$. By the deformation lemma, there exist $\varepsilon>0$ with $\bar{c}+\varepsilon<0$ and an odd homeomorphism $\eta:E\to E$ such that
\[
\eta(\tilde{I}^{-1}((-\infty,\bar{c}+\varepsilon])\setminus N_\delta (\mathcal{C}_{\tilde{I}, \bar{c}}))\subset \tilde{I}^{-1}((-\infty,\bar{c}-\varepsilon]).
\]
By the assumption that $c_n\to \bar{c}$ and the fact that $\{c_n\}$ is non-decreasing, there exists $n\in\mathbb{N}$ such that $c_n>\bar{c}-\varepsilon$ and $c_{\gamma(\mathcal{C}_{\tilde{I}, \bar{c}})+n}\leq \bar{c}$. We can take $A\in\Sigma_{\gamma(\mathcal{C}_{\tilde{I}, \bar{c}})+n}$ such that 
\[
\displaystyle\sup_A \tilde{I}<\bar{c}+\varepsilon.
\]
Then, we have
\[
\gamma(\overline{A\setminus N_\delta (\mathcal{C}_{\tilde{I}, \bar{c}})})=\gamma(A)-\gamma(N_\delta (\mathcal{C}_{\tilde{I}, \bar{c}}))= \gamma(A)-\gamma(\mathcal{C}_{\tilde{I}, \bar{c}})\geq n,
\]
and since $\eta$ is an odd homeomorphism, we also have 
\[
\gamma(\eta(\overline{A\setminus N_\delta (\mathcal{C}_{\tilde{I}, \bar{c}})}))= \gamma(\overline{A\setminus N_\delta (\mathcal{C}_{\tilde{I}, \bar{c}})})=n,
\]
that is, $\eta(\overline{A\setminus N_\delta (\mathcal{C}_{\tilde{I}, \bar{c}})}) \in\Sigma_n$. 
\[
\bar{c}-\varepsilon<c_n\leq \sup_{\eta(\overline{A\setminus N_\delta (\mathcal{C}_{\tilde{I}, \bar{c}})})}\tilde{I}
\]
which contradicts
\[
\eta(\overline{A\setminus N_\delta (\mathcal{C}_{\tilde{I}, \bar{c}})})\subset \eta(\tilde{I}^{-1}((-\infty,\bar{c}+\varepsilon])\setminus N_\delta (\mathcal{C}_{\tilde{I}, \bar{c}}))\subset \tilde{I}^{-1}((-\infty,\bar{c}-\varepsilon]).
\]
Therefore, we deduce $c_n\to 0$. In addition, for any critical point $u$ of $\tilde{I}$ with $\tilde{I}[u]=c_n$, since $\tilde{I}[u] <0$, we have $I[u]= \tilde{I}[u]$ and thus $u$ is also a critical point of $I$.
\end{proof}

\section{Extensions and Variants}
We can extend the result to the non-degenerate Choquard-Kirchhoff type problem whose associated energy functional is of the form
\begin{align*}
I_{\mathcal{M}}[u] &=\frac{1}{p}\mathcal{M}(\|u\|_{D^{s,p}}^p)+ \frac{1}{p}\int_{\mathbb{R}^N}V(x)|u|^p dx \\
&\phantom{=} -\frac{1}{2\cdot p_{r;s}^{\uparrow *}} \varepsilon_K\int_{\mathbb{R}^N}(K\ast g(u))g(u)dx-\varepsilon_W\int_{\mathbb{R}^N}W(x)f(u)dx.
\end{align*}
where $\mathcal{M} \in C^1 (\mathbb{R}_{\geq 0},\mathbb{R})$ satisfies the following conditions:
\begin{itemize}
\item[(M1)] $M\coloneqq\mathcal{M}'\in C(\mathbb{R}_{\geq 0},\mathbb{R})$ and $\inf M>0$.
\item[(M2)] There exists $\theta\in [1,(2N-N/r)/(N-ps))$ such that $\theta\mathcal{M}(t)\geq t M(t)$ ($\forall t\geq 0$). 
\end{itemize}
We can adopt the method in \cite{Kirchhoff} and combine it with our method to obtain a result similar to Theorem \ref{MainTheorem} as follows in this case, though we omit the proof. 
\begin{theorem}
Assume (M1) and (M2) in addition to the assumptions of Theorem \ref{MainTheorem}. Then,
\begin{enumerate}
\item for each $\varepsilon_K>0$, there exists $\beta_0>0$ such that if $0<\varepsilon_W<\beta_0$, then $I_{\mathcal{M}}$ has a sequence $\{u_n\}$ of critical points with $I_{\mathcal{M}}[u_n]<0$, $I_{\mathcal{M}}[u_n]\to 0$ and $u_n\to 0$ ($n\to\infty$).
\item for each $\varepsilon_W>0$, there exists $\lambda_0>0$ such that if $0<\varepsilon_K<\lambda_0$, then $I_{\mathcal{M}}$ has a sequence $\{u_n\}$ of critical points with $I_{\mathcal{M}}[u_n]<0$, $I_{\mathcal{M}}[u_n]\to 0$ and $u_n\to 0$ ($n\to\infty$).
\end{enumerate}
\end{theorem}

Also note that the same conclusion remains true with inhomogeneous convolution nonlinearity $g(x,u)$ instead of $g(u)$ under the following assumptions (G1'), (G2') instead of (G1), (G2).
\begin{enumerate}
\item[(G1')] $h(x,-u)=h(x,u)$, $h(x,0)=0$ and $|h_u(x,u)|\leq C(|u|^{\hat{p}_g-1}+|u|^{p_g-1})$ ($\forall x\in \mathbb{R}^N$, $\forall u\in\mathbb{R}$) for some $p_g,\hat{p}_g$ with $p<\hat{p}_g\leq p_g<p_{r;s}^{\uparrow *}$ and $C>0$ where $h(x,t)\coloneqq g(x,t)-|t|^{p_{r;s}^{\uparrow *}}$.
\item[(G2')] $0<\alpha_g h(x,u)\leq u h_u (x,u)$ ($\forall x\in \mathbb{R}^N$, $\forall u\neq 0$) for some $\alpha_g\in (p_{r;s}^{\downarrow *}, p_{r;s}^{\uparrow *}]$. 
\end{enumerate}

\end{document}